\documentclass[12pt]{amsart}
\usepackage{amssymb}
\usepackage{multicol}
\usepackage{colordvi}
\usepackage{graphicx}

\usepackage{mathrsfs}

\usepackage{vmargin}
\setpapersize{USletter}
\setmargrb{2.5cm}{2.5cm}{2.5cm}{2.5cm} 

\newtheorem{theorem}{Theorem}[section]
\newtheorem{proposition}[theorem]{Proposition}
\newtheorem{lemma}[theorem]{Lemma}

\theoremstyle{definition}

\newtheorem{example}[theorem]{Example}
\newtheorem{remark}[theorem]{Remark}

\copyrightinfo{}{}
\newcounter{FNC}[page]
\def\fauxfootnote#1{{\addtocounter{FNC}{2}\Magenta{$^\fnsymbol{FNC}$}%
     \let\thefootnote\relax\footnotetext{\Magenta{$^\fnsymbol{FNC}$#1}}}}
\renewcommand{\qed}{\hfill\raisebox{-0.2pt}{\includegraphics{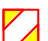}}}
\newcommand{\QED}{{\includegraphics{figures/QED.eps}}}

\newcommand{\relint}{\mathop{\rm relint}\nolimits}
\newcommand{\spec}{\mathop{\rm spec}\nolimits}

\newcommand{\pls}{\mathscr{P}^{\infty}}
\newcommand{\scrA}{{\mathscr{A}}}
\newcommand{\scrC}{{\mathscr{C}}}
\newcommand{\scrM}{{\mathscr{M}}}
\newcommand{\nca}{{\mathit{Nco}\mathscr{A}}}
\newcommand{\coscrA}{{\mathit{co}\mathscr{A}}}
\newcommand{\scrP}{{\mathscr{P}}}
\newcommand{\scrL}{{\mathscr{L}}}

\newcommand{\calA}{{\mathcal{A}}}
\newcommand{\calF}{{\mathcal{F}}}

\newcommand{\fraka}{{\mathfrak a}}
\newcommand{\frakm}{{\mathfrak m}}

\newcommand{\bm}{{\bf m}}
\newcommand{\bn}{{\bf n}}

\renewcommand{\P}{{\mathbb P}}
\newcommand{\C}{{\mathbb C}}
\newcommand{\K}{{\mathbb K}}
\newcommand{\N}{{\mathbb N}}
\newcommand{\Q}{{\mathbb Q}}
\newcommand{\R}{{\mathbb R}}
\newcommand{\T}{{\mathbb T}}
\newcommand{\U}{{\mathbb U}}
\newcommand{\Z}{{\mathbb Z}}
\newcommand{\bS}{{\mathbb S}}

\newcommand{\ini}{\mbox{\rm in}}
\newcommand{\nvol}{\mbox{\rm n-vol}}
\newcommand{\rank}{\mathop{\rm rank}\nolimits}
\newcommand{\sat}{\mbox{\rm sat}}

\newcommand{\Log}{\mathop{\rm Log}\nolimits}
\newcommand{\trop}{\mathop{\rm Trop}\nolimits}
\newcommand{\Trop}{\mathop\mathfrak{Trop}\nolimits}
\newcommand{\Ptrop}{\mathop{\scrP{\rm trop}}\nolimits}
\newcommand{\Arg}{\mathop{\rm Arg}\nolimits}
\newcommand{\Hom}{\mathop{\rm Hom}\nolimits}

\newcommand{\defcolor}[1]{\Blue{#1}}
\newcommand{\demph}[1]{\defcolor{{\sl #1}}}

\title{Non-archimedean coamoebae}
\author{Mounir Nisse}
\address{Mounir Nisse\\
          Department of Mathematics\\
         Texas A\&M University\\
         College Station\\
         Texas \ 77843\\
         USA}
\email{nisse@math.tamu.edu}
\urladdr{www.math.tamu.edu/\~{}nisse}
\author{Frank Sottile}
\address{Frank Sottile\\
         Department of Mathematics\\
         Texas A\&M University\\
         College Station\\
         Texas \ 77843\\
         USA}
\email{sottile@math.tamu.edu}
\urladdr{www.math.tamu.edu/\~{}sottile}
\thanks{Research of Sottile is supported in part by NSF grant DMS-1001615 and the Institut
  Mittag-Leffler.} 
\subjclass{14T05, 32A60} 
\begin{document}

\begin{abstract}
 A coamoeba is the image of a subvariety of a complex torus under the argument map
 to the real torus.
 Similarly, a non-archimedean coamoeba is the image of a subvariety of a torus over a
 non-archimedean field $\K$ with complex residue field under an argument map.
 The phase tropical variety is the closure of 
 the image under the pair of maps, tropicalization and
 argument.
 We describe the structure of non-archimedean coamoebae and phase tropical varieties 
 in terms of complex coamoebae and their phase limit sets.
 The argument map depends upon a section of the 
 valuation map, and we explain how this choice (mildly) affects the non-archimedean
 coamoeba.
 We also identify a class of varieties whose non-archimedean coamoebae and phase tropical
 varieties are objects from polyhedral combinatorics.
\end{abstract}

\maketitle

\begin{center}
Dedicated to the memory of Mikael Passare 1.1.1959---15.9.2011.
\end{center}

\section*{Introduction}

A coamoeba is the image of a subvariety $X$ of a complex torus $(\C^*)^n$ under the argument
map to the real torus $\defcolor{\U^n}:=\R^n/(2\pi\Z)^n$.
It is related to the amoeba, which is the image of $X$ under
the coordinatewise logarithm map $z\mapsto\log|z|$ to $\R^n$.
Amoebae were introduced by Gelfand, Kapranov, and Zelevinsky in 1994~\cite{GKZ}, and have
subsequently been widely studied~\cite{KeOk,Mi,PaRu,Pur}.
In contrast, coamoebae, which were introduced by Passare in a talk in 2004, have
only just begun to be studied and they appear to have many beautiful and interesting
properties. 
For example, the closure of the coamoebae of $X$ is the (finite) union of the coamoebae of
all its initial schemes, $\ini_wX$.
Those for $w\neq 0$ form the phase limit set~ of $X$~\cite{NS}, whose structure is controlled by
the tropical variety of $X$.

Suppose that $\K$ is an algebraically closed field with a non-trivial valuation 
$\nu\colon\K^*\to\Gamma$, where $\Gamma$ is a totally ordered divisible abelian group.
A variety $X$ in $(\K^*)^n$ has a non-archimedean amoeba, or tropicalization,
$\trop(X)\subset\Gamma^n$~\cite{IMS}. 
When $\K$ has residue field the complex numbers, there a version of the
argument map to $\U^n$ defined on $\K^*$, and the image of $X\subset(\K^*)^n$ under this
argument map is a \demph{non-archimedean coamoeba}.
Non-archimedean coamoebae have structure controlled by the tropical variety of $X$, which 
generalizes the structure of the phase limit set of a complex variety.

The closure of the image of $X$ in $\Gamma^n\times\U^n$ under the product of the
tropicalization and argument maps is the \demph{phase tropical variety} of $X$,
\defcolor{$\Ptrop(X)$}. 
For curves over the field of Puiseaux series, this notion was introduced and studied by
Mikhalkin~\cite[\S6]{Mik05}. 

These objects depend upon a section $\tau\colon\Gamma\to\K^*$ of the valuation map,
which allows the identification of all reductions of fibers of tropicalization 
$(\K^*)^n\to\Gamma^n$ with $(\C^*)^n$.
The identifications of a fiber over $w\in\Gamma^n$ for two different sections are
related by a translation $\alpha(w)\in(\C^*)^n$. 
Translations in different fibers can be significantly different.
For example, when 
$\dim_\Q\Gamma>1$, the map $\alpha\colon\Gamma^n\to(\C^*)^n$
may not be continuous.
We explain the (surprisingly mild) effect of a change of section on these objects.

A variety is tropically simple if its tropical reductions are pullbacks along surjections of
products of hyperplanes. 
The non-archimedean coamoeba of a tropically simple variety is essentially an object of
polyhedral combinatorics and may be realized as 
a union of products of zonotope complements in subtori glued together
along their phase limit sets.
The phase tropical variety of tropically simple variety has a similar description. 

In Section~\ref{S:defs}, we review the fundamental theorems of tropical geometry in a form that
we need, including a discussion of how the choice of section of valuation affects tropical
reductions, and review coamoebae.
We define non-archimedean coamoebae and phase tropical varieties in Section~\ref{S:main} and
establish their basic structure.
We close with Section~\ref{S:simple} which studies these objects for tropically simple
varieties. 

We thank Diane Maclagan, Sam Payne, Jan Draisma, and Brian Osserman who helped us
to understand the structure of tropical reductions.
We also posthumously thank Mikael Passarre for his friendship, guidance, and encouragement
of this line of research.

\section{Tropicalization and coamoebae}\label{S:defs}

Let \defcolor{$\K$} be an algebraically closed field equipped with a non-archimedean valuation
$\defcolor{\nu}\colon\K^*\twoheadrightarrow\Gamma$, where \defcolor{$\Gamma$} is a
divisible totally ordered group. 
We do not assume that $\Gamma$ is a subgroup of $\R$, for example
$\Gamma=\R^2$ with lexicographic order is suitable.
We extend $\nu$ to $\K$ by setting $\nu(0)=\infty$, which exceeds every
element of $\Gamma$.
Then
\[
   \nu(ab)\ =\ \nu(a)+\nu(b)
   \qquad\mbox{and}\qquad
   \nu(a+b)\ \geq\ \min\{\nu(a),\nu(b)\}\,,
\]
with equality when $\nu(a)\neq\nu(b)$.
Write $R$ for the value ring and $\frakm$ for its maximal ideal,
\[
   \defcolor{R}\ :=\ \{x\in\K\mid \nu(x)\geq 0\}
    \qquad\mbox{and}\qquad
   \defcolor{\frakm}\ :=\ \{x\in\K\mid \nu(x)> 0\}\,.
\]
The residue field of $\K$ is $\defcolor{\Bbbk}:=R/\frakm$, which is also algebraically closed.

%
\subsection{Tropicalization}

Let \defcolor{$N$} be a finitely generated free abelian group and 
$\defcolor{M}:=\Hom(N,\Z)$ its dual group. 
Write \defcolor{$\langle\cdot,\cdot\rangle$} for the pairing between $M$ and $N$.
The group ring \defcolor{$\Z[M]$} is the coordinate ring of a torus $\T_N$ whose
$\K$-points are $\K^*\otimes_\Z  N = \Hom(M,\K^*)$, the set of group homomorphisms
$M\to\K^*$.  
The valuation induces a map 
$\defcolor{\trop}\colon\T_N(\K)\to\defcolor{\Gamma_N}:=\Gamma\otimes_\Z N$.
Maps $N'\to N$ of free abelian groups functorially induce maps $\T_{N'}\to\T_N$ and
$\Gamma_{N'}\to\Gamma_N$. 

The points $x\in\T_N(\K)$ with valuation $w$, 
\[
   \{ x\colon M\to\K \mid \nu(x(\bm))=\langle\bm,w\rangle\}\,,
\]
form the fiber of $\T_N(\K)$ over $w\in\Gamma_N$.
For $w=0$ this is $\T_N(R)$, which has a natural reduction map onto $\T_N(\Bbbk)$.
We follow Payne~\cite{Payne} to understand the other fibers. 
Write \defcolor{$\xi^\bm$} for the element of $\K[M]$ corresponding to $\bm\in M$.
The \demph{tilted group ring} is
\[
   \defcolor{R[M]^w}\ :=\ \Bigl\{ \sum c_\bm \xi^\bm\ \in\K[M] \mid  
        \nu(c_\bm)+\langle\bm,w\rangle\geq 0\Bigr\}\,.
\]
If $\defcolor{\T^w}:=\spec R[M]^w$, then $\T^w(R)$ is the fiber of $\T_N(\K)$ over
$w\in\Gamma_N$. 
The residue map $\T^w(R)\to\T^w(\Bbbk)$ associates a point $x\in\T^w(R)$
canonically to its \demph{tropical reduction} $\defcolor{\overline{x}}\in\T^w(\Bbbk)$.

The map $c\,\xi^\bm\mapsto \xi^\bm\otimes c\,\xi^\bm$ induces a coaction 
$R[M]^w\to R[M]\otimes R[M]^w$, and so $\T^w$ is a $\T_N$-torsor over $R$.
Thus these fibers $\T^w(R)$ and $\T^w(\Bbbk)$ are non-canonically isomorphic to the tori
$\T_N(R)$ and $\T_N(\Bbbk)$.
We also have the \demph{exploded tropicalization} map 
\[
     \defcolor{\Trop}\ \colon\ \T_N(\K)\ \longrightarrow\ 
     \coprod_{w\in \Gamma_N}\T^w(\Bbbk)\,.
\]
This union is disjoint because there is no natural identification between the fibers of
$\trop$.

The \demph{tropicalization} $\defcolor{\trop(X)}\subset\Gamma_N$ of a
subscheme $X\subset\T_N$ is the image of $X(\K)$ under $\trop$.
Its \demph{exploded tropicalization $\Trop(X)$} is its image under $\Trop$.
The fiber of $\Trop(X)$ over a point $w\in\trop(X)$ is the 
\demph{tropical reduction $X_w$} of $X$, which is a subset of $\T^w(\Bbbk)$.
We consider this at the level of schemes.

The fiber of the $R$-scheme $\T^w$ over the closed point of $\spec R$ is a scheme
over the residue field $\Bbbk$ with coordinate ring 
$\defcolor{\Bbbk[M]^w}:= R[M]^w\otimes_R \Bbbk$.
This has a concrete description.
For $\gamma\in\Gamma$, let \defcolor{$\fraka^\gamma$} be the fractional ideal of $R$
consisting of those $x\in\K$ with valuation $\nu(x)\geq\gamma$.
Set
$\defcolor{\Bbbk^\gamma}:=\fraka^\gamma/(\frakm\cdot\fraka^\gamma)=\fraka^\gamma\otimes_R\Bbbk$, 
which is a one-dimensional $\Bbbk$-vector space.
Then
\[
   R[M]^w\ =\ \bigoplus_{\bm\in M} \fraka^{-\langle\bm,w\rangle}\xi^\bm
     \qquad\mbox{and}\qquad
   \Bbbk[M]^w\ =\ \bigoplus_{\bm\in M} \Bbbk^{-\langle\bm,w\rangle}\xi^\bm\,.
\]

Each point $w\in\Gamma_N$ determines a weight function on monomials in $\K[M]$,
\[
   c\,\xi^\bm\ \longmapsto\ \nu(c)+\langle\bm,w\rangle\,,
\]
for $c\in\K^*$ and $\bm\in M$.
If all terms of $f\in\K[M]$ have non-negative $w$-weight, then $f\in R[M]^w$,
and its image in $\Bbbk[M]^w$ is its \demph{tropical reduction $f_w$}.
The \demph{tropical reduction $I_w$} of an ideal $I\subset\K[M]$ is
\[
   I_w\ := \{ f_w\mid f\in I\cap R[M]^w\}\ \subset\ \Bbbk[M]^w\,.
\]
Surjectivity of tropicalization relates tropical reductions tropicalization.

\begin{proposition}\label{Prop:surjective}
 Suppose that $X$ is a $\K$-subscheme of\/ $\T_N$ with ideal $I\subset\K[M]$
 and let $w\in\Gamma_N$.
 Its tropical reduction $X_w$ is the $\Bbbk$-subscheme of\/ $\T^w$ defined by
 $I_w$. 
 This is nonempty exactly when $w\in\trop(X)$.
\end{proposition}

This, as well as Propositions~\ref{Prop:FTTG},~\ref{Prop:ini}, and~\ref{Prop:local} are the
fundamental theorems of tropical geometry, and are due to many authors.
For a discussion with references, see Section 2.2 of~\cite{OP}.
Another complete source presented in refreshing generality is~\cite{Gubler}.

%
\subsection{Structure of tropicalization}\label{S:polyhedron}

The usual notions of rational polyhedral complexes in $\R^n$ carry over
without essential change to $\Gamma_N$.
(But when $\Gamma\not\subset\R$, they are not easily interpreted as ordinary polyhedral
complexes~\cite{Ban11}.) 
For example, a (\demph{$\Gamma$-rational}) \demph{polyhedron} 
is a subset $\sigma$ of $\Gamma_N$ which is the intersection of finitely many half-spaces of
the form 
 \begin{equation}\label{Eq:half-space}
   \{w \in \Gamma_N\ \mid\ \langle\bm,w\rangle\ \geq\ a\}\,,
 \end{equation}
where $a\in\Gamma$ and $\bm\in M$.
If all constants $a$ in~\eqref{Eq:half-space} are $0$, then $\sigma$ is a \demph{cone}.
A \demph{face} of a polyhedron $\sigma$ is either $\sigma$ itself or the intersection of
$\sigma$ with a hyperplane $\{ w \mid \langle\bm,w\rangle = a\}$, 
where  $a,\bm$ define a half-space~\eqref{Eq:half-space} containing $\sigma$.
The \demph{relative interior}, \demph{$\relint(\sigma)$} of a polyhedron $\sigma$ is the
set-theoretic difference of $\sigma$ and its proper faces.

The intersection of all hyperplanes containing the differences of elements in a
polyhedron $\sigma$ is a subgroup of $\Gamma_N$ which 
has the form $\Gamma_{N'}$ for a sublattice $N'$ of $N$, written as $\langle\sigma\rangle$.
The rank of this sublattice $\langle\sigma\rangle$ is the \demph{dimension} of $\sigma$.
For any $w\in\sigma$, we have $\sigma\subset w+\Gamma_{\langle\sigma\rangle}$.

A \demph{polyhedral complex $\scrC$} is a finite union of polyhedra, called \demph{faces} of
$\scrC$, such that if $\sigma\in\scrC$, then all faces of $\sigma$ lie in $\scrC$, and if
$\sigma,\sigma'\in\scrC$, then either $\sigma\cap\sigma'$ is empty or it is a common face of
both polyhedra. 
A \demph{fan} is a polyhedral complex in which all faces are cones.
A polyhedral complex has \demph{pure dimension $d$} if every maximal face has
dimension $d$.
The \demph{support} of a polyhedral complex is the union of its faces.

The minimal face containing a point $w$ in the support of $\scrC$ 
is the unique face $\sigma$ with $w\in\relint(\sigma)$.
For every face $\rho$ of $\scrC$ containing $w$, there is a cone $\rho_w$ such that 
$\rho_w+w$ is the intersection of those half-spaces containing $\rho$ whose boundary
hyperplane contains $w$.  
The union of these cones $\rho_w$ for faces $\rho$ of $\scrC$ containing $w$ forms the
\demph{local fan $\scrC_w$} of $\scrC$ at $w$.
These cones all contain the subgroup $\Gamma_{\langle\sigma\rangle}$.

We state the basic structure theorem of tropical varieties.

\begin{proposition}\label{Prop:FTTG}
  Let $X$ be an irreducible $\K$-subscheme of\/ $\T_N$ of dimension $d$.
  Then  $\trop(X)$ is the support of a connected polyhedral complex 
  $\scrC$ in $\Gamma_N$  of pure dimension $d$.
\end{proposition}

\begin{remark}\label{R:scrC}
 We may partially understand the polyhedral complex $\scrC$ as follows.
 A polynomial $f\in\K[M]$ may be written
 \begin{equation}\label{Eq:polyF}
  f\ =\ \sum_{\bm\in\calA(f)} c_{\bm}\xi^\bm
  \qquad c_\bm\in\K^*\,,
 \end{equation}
 where the support $\calA(f)$ of $f$ is a finite subset of $M$.
 For $w\in\Gamma_N$, let \defcolor{$\calA(f)_w$} be the points of $\calA(f)$ indexing terms
 of~\eqref{Eq:polyF} with minimum $w$-weight, $\nu(c_\bm)+\langle \bm,w\rangle$.
 When this common minimum is $0\in\Gamma$, the tropical reduction of $f$ is
\[
    f_w\ =\ \sum_{\bm\in\calA(f)_w} \overline{c_{\bm}}\xi^\bm\ \
    \in\ \Bbbk[M]^w\,.
\]

 The polyhedral complex $\scrC$ depends upon 
 an embedding of $\T_N$ as the dense torus in a projective space $\P^N$.
 This corresponds to a choice 
 of generators of $M$, which gives a polynomial subalgebra $\K[\N^n]$ 
 ($n=\rank(N)$) of $\K[M]$.
 Given a $\K$-subscheme $X\subset\T_N$ defined by an ideal $I\subset\K[M]$, 
 let $\widetilde{I}$ be the homogenization of $I\cap \K[\N^n]$, which defines the 
 closure of $X$ in $\P^N$.
 Then polyhedra $\sigma\in\scrC$ have the property that for $w,u\in\relint(\sigma)$, 
 $\calA(f)_w=\calA(f)_u$, for every $f\in\widetilde{I}$.
 This implies the same for every $f\in I$.
\qed
\end{remark}

 We next relate tropical reductions to 
 the faces of $\scrC$.
 This requires that we compare tropical reductions at different points of $\Gamma_N$.

%
\subsection{Trivializations of tropicalization}

The fibers $\T^w(\Bbbk)$ of tropicalization may be identified so that $\Trop(X)$
becomes a subset of $\Gamma_N\times\T^w(\Bbbk)$---but at the price of a choice
of a homomorphism $\tau\colon\Gamma\to\K^*$ that is a section of the valuation map.
As David Speyer pointed out, such sections always exist.
In the exact sequence
\[
   \{1\}\ \longrightarrow\ \ker(\nu)\ \longrightarrow\ 
    \K^*\ \stackrel{\nu}{\longrightarrow}\ \Gamma\ \longrightarrow\ \{0\}\,,
\]
the group $\ker(\nu)$ is divisible as $\K$ is algebraically closed.
Therefore $\nu$ has a section.

A section $\defcolor{\tau}\colon\Gamma\to\K^*$ of the valuation map induces a section
of tropicalization $\defcolor{\tau_N}\colon\Gamma_N\to\T_N(\K)$ via
$\tau_N(w)\colon\bm\mapsto\tau(\langle \bm,w\rangle)$.
This is a homomorphism, $\tau_N(w+u)=\tau_N(w)\tau_N(u)$.
Write \defcolor{$\tau^w$} for $\tau_N(w)$, which is an element of $\T^w(R)$.
We have an identification
\[
  \varphi_{\tau,w}\ \colon\ \T_N(R)\ \xrightarrow{\ \sim\ }\ \T^w(R)
   \qquad\mbox{via}\qquad 
   x\ \longmapsto\ x\cdot \tau^w\,.
\]
Similarly, $\overline{x}\mapsto \overline{x}\cdot\overline{\tau^w}$ 
identifies $\T_N(\Bbbk)$ with $\T^w(\Bbbk)$.
We call this a \demph{trivialization of tropicalization}.
We obtain the identification of exploded tropicalization
 \begin{equation}\label{E:trivialization}
   \varphi_\tau^{-1}\ \colon\ 
   \coprod_{w\in\Gamma_N}T^w(\Bbbk)\  \xrightarrow{\ \sim\ }\ \Gamma_N\times \T_N(\Bbbk)
   \qquad\mbox{via}\qquad 
    T^w(\Bbbk)\ni\overline{x}\ \longmapsto\ 
     (w,\overline{\tau^{-w}}\cdot\overline{x})\,.
 \end{equation}

Let us fix a section $\tau\colon\Gamma\to\K^*$, and use it to
identify all tropical reductions $T^w(\Bbbk)$ with $\T_N(\Bbbk)$.
That is, we set $\defcolor{X_{w,\tau}}\subset\T_N(\Bbbk)$ to be 
$\overline{\tau^{-w}}\cdot X_w$.
A consequence of this choice is the following proposition. 

\begin{proposition}\label{Prop:ini}
  Let $X\subset\T_N$ be a $\K$-subscheme and suppose that $\scrC$ is a polyhedral complex as
  in Remark~$\ref{R:scrC}$.
  For any face $\sigma\in\scrC$, the tropical reductions $X_{w,\tau}$ and $X_{u,\tau}$ are
  equal, for any $w,u\in\relint(\sigma)$.
\end{proposition}

We omit the section $\tau$ in our notation when it is understood.
For $\sigma$ a face of $\scrC$, we write \defcolor{$X_\sigma$} (or $X_{\sigma,\tau}$) for this
common tropical reduction at points of $\relint(\sigma)$.

\begin{remark}
 We explain this at the level of schemes.
 Let $I\subset\K[M]$ be the ideal of $X$ and 
 suppose that $f\in I$ as in~\eqref{Eq:polyF} with support $\calA(f)$.
 If $w(f)$ is the minimal $w$-weight of a term of $f$, then the polynomial
 $\tau(-w(f))\cdot f$ lies in $I\cap R[M]^w$ and
\[
   \overline{\varphi^*_{\tau,w}(\tau(-w(f))\cdot f)}\ =\ 
   \sum_{\bm\in\calA(f)_w} \overline{c_\bm}\cdot
       \overline{\tau(\langle\bm,w\rangle-w(f))}\cdot \xi^\bm\ \in\ \Bbbk[M]
\]
 lies in the ideal of $X_{w,\tau}=\overline{\tau^{-w}}\cdot X_w\subset\T_N(\Bbbk)$.
 If $\calA(f)_w=\calA(f)_u$, then an easy calculation shows that
 $\langle\bm,w\rangle-w(f)=\langle\bm,u\rangle-u(f)$, and thus
\[
   \overline{\varphi^*_{\tau,w}(\tau(-w(f))\cdot f)}\ =\ 
   \overline{\varphi^*_{\tau,u}(\tau(-u(f))\cdot f)}.
\]
 As faces $\sigma$ of $\scrC$ have the property that for any $w,u\in\relint(\sigma)$
 and any $f\in I$, we have  $\calA(f)_w=\calA(f)_u$, this explains Proposition~\ref{Prop:ini}. 
\qed
\end{remark}

%
\subsection{Local structure of tropicalization}

Let $K$ be an algebraically closed field.
Define the initial form $\ini_w f$, initial ideal $\ini_w I$, and initial scheme $\ini_w
X$ for $f\in K[M]$, $I$ an ideal of $K[M]$, and $X$ a $K$-subscheme of $\T_N$ as in the theory
of Gr\"obner bases~\cite{GBCP}. 
For example, $\ini_w f$ is the sum of terms of $f$ on whose exponents the
function $\bm\mapsto\langle\bm,w\rangle$ is minimized.

The \demph{tropicalization $\trop(X)$} of a $K$-subscheme of $\T_N$ is the set of 
$w\in\Gamma_N$ such that $\ini_w X$ is non-empty.
This corresponds to the usual tropicalization of $X(\K)$, where $\K$ is any non-archimedean
valued field containing $K$ whose valuation extends the trivial valuation on $K$ and whose
value group is $\Gamma$.
Choosing a projective space closure $\P^N$ of $\T_N$, $\trop(X)$ is the support of a
rational polyhedral fan in $\Gamma_N$, which is a subfan of the Gr\"obner fan of the
closure of $\overline{X}$. 
Call this subfan $\Sigma$ a \demph{Gr\"ober-tropical fan} of $X$.
All elements in the relative interior of a cone $\sigma$ of $\Sigma$ induce the same initial
scheme, which we write as $\ini_\sigma X$.\smallskip

Let us return to the situation of the previous subsections.
That is, $X$ is a $\K$-subscheme of $\T_N$ and $\K$ is a field with non-archimedean valuation
$\nu$, value group $\Gamma$, and residue field $\Bbbk$.
Fix a section $\tau$ of the valuation map and a projective space closure $\P^N$ of $\T_N$.
Then there is a polyhedral complex $\scrC$ in $\Gamma_N$ whose support is $\trop(X)$.
For $w\in\trop(X)$, consider the local fan $\scrC_w$ of $\scrC$ at $w$.
The tropical reduction $X_w$ is a $\Bbbk$-subscheme of $\T_N$.
Equipping $\Bbbk$ with a
trivial valuation to $\Gamma$, the  Gr\"obner-tropical fan $\Sigma$ of $X_w$ is a second fan
associated to $w\in\trop(X)$, and its support is the tropicalization $\trop(X_w)$  of $X_w$.

\begin{proposition}\label{Prop:local}
  We have $\scrC_w=\Sigma$.
  For any polyhedron $\rho$ of $\scrC$ containing $w$, we have
\[
    X_\rho \ =\ \ini_{\rho_w}(X_w)\,.
\]
\end{proposition}

Here, $X_\rho$ is the tropical reduction of $X$ at any point in the relative interior of
$\rho$, while  $\ini_{\rho_w}(X_w)$ is the initial scheme of the tropical reduction $X_w$
corresponding to the cone $\rho_w$ of the local fan $\scrC_w$  at $w$, equivalently, the
Gr\"obner-tropical fan of $X_w$.

Lastly, we recall a fact from the theory of Gr\"obner bases.
Suppose that $\sigma$ is a cone of the Gr\"obner fan of a projective scheme $Y$.
Then $\langle\sigma\rangle$ is a sublattice of $N$ of rank equal to the dimension of $\sigma$,
and the corresponding torus $\T_{\langle\sigma\rangle}\subset\T_N$ acts on $\ini_\sigma Y$.
In particular, when $X\subset\T_N$ is irreducible of dimension $d$ and $\sigma$ is a
$d$-dimensional polyhedron in a polyhedral complex $\scrC$ representing $\trop(X)$, then 
the tropical reduction $X_\sigma$ is a $d$-dimensional subscheme of $\T_N$ on which 
the $d$-dimensional subtorus $\T_{\langle\sigma\rangle}$ acts, 
so that $X_\sigma$ is a finite union of $\T_{\langle\sigma\rangle}$-orbits.

%
\subsection{Change of trivialization}\label{S:change}

We now examine the effect on the tropical reductions of a different choice of section
$t\colon\Gamma\to\K^*$ of the valuation homomorphism. 
This differs from $\tau$ by a group homomorphism 
$\defcolor{\alpha}:=t/\tau\colon\Gamma\to \defcolor{R^*}:=R\setminus\frakm$.
This gives a homomorphism $\alpha_N\colon \Gamma_N\to\T_N(R)$ with value
$\alpha^w:=\alpha_N(w)\in\T_N(R)$ at $w\in \Gamma_N$ defined by $t^w=\alpha^w\cdot\tau^w$.
Trivializations induced by $t$ and $\tau$ in $\T^w(\Bbbk)$ are translates of each
other by the reduction $\overline{\alpha^w}$ of $\alpha^w$.
Conversely, any homomorphism $\alpha\colon\Gamma\to R^*$ gives a section
$\alpha\cdot\tau\colon\Gamma\to\K^*$.

\begin{example}\label{ex:pathology}
 Suppose that $\Gamma=\R$ and the residue field is $\C$, which is a
 subfield of $\K$.
 Then any homomorphism $\alpha\colon\R\to\C$ gives a
 section $\alpha\cdot\tau$ of the valuation.
 Suppose that $N=\Z^n$, then the two identifications~\eqref{E:trivialization} of exploded
 tropicalization with $\R^n\times(\C^*)^n$ differ by the map
\[
   \R^n\times(\C^*)^n\ \ni\ (w,x)\ \longmapsto\ 
   (w,\overline{\alpha^w}\cdot x)\ \in\ \R^n\times(\C^*)^n\,.
\]
 When $\alpha$ is discontinuous, this will be discontinuous, and so it is not 
 clear that any reasonable structure can arise from the choice of a section and the
 corresponding trivialization. \qed
\end{example}

Despite this, the different trivializations behave reasonably 
well on fibers of tropicalization.

Let $X\subset\T_N$ be a $\K$-subscheme.
For any $w\in\Gamma_N$, the two tropical reductions are related by the translation
 \begin{equation}\label{Eq:translation}
   X_{w,\tau}\ =\ \overline{\alpha^w}\cdot X_{w,t}\,.
 \end{equation}
Fix a polyhedral complex $\scrC$ supported on $\trop(X)$ as in Remark~\ref{R:scrC}.
For any face $\sigma$ of $\scrC$ containing $w$, we have 
$\sigma\subset w+\Gamma_{\langle\sigma\rangle}$, and so
\[
   \overline{\alpha}\ \colon\ \sigma\ \longrightarrow\ 
   \overline{\alpha^w}\cdot\overline{\alpha}(\Gamma_{\langle\sigma\rangle})
   \ =\ \overline{\alpha^w}\cdot\T_{\langle\sigma\rangle}(\Bbbk)\,,
\]
 as $\overline{\alpha}\colon \Gamma_{\langle\sigma\rangle}\to\T_{\langle\sigma\rangle}(\Bbbk)$,
by the functoriality of our constructions.
Since $\T_{\langle\sigma\rangle}(\Bbbk)$ acts freely on $X_\sigma$, we see that 
\[
  X_{\sigma,\tau}\ =\ \overline{\alpha^w}\cdot X_{\sigma,t}
\]
for any $w\in\sigma$ (not just those in the relative interior of $\sigma$, which is implied
by~\eqref{Eq:translation}). 

Thus the two tropical reductions for every face containing $w\in\trop(X)$
are related by the same translation.
In particular, the tropical reductions for incident faces are related by the same
translation, but this is not the case for non-incident faces.
Thus we only need to know the translations at minimal faces of $\scrC$.
Let $\scrM(\scrC)$ be the set of minimal faces of $\scrC$.
For each $\sigma\in\scrM(\scrC)$, we may choose some point $w\in\sigma$ and set
$a_\sigma:=\overline{\alpha^w}$. 
We summarize some conclusions of this discussion.

\begin{proposition}\label{P:trivialization}
 With these definitions, if $w\in\trop(X)$, then 
\[
   X_{w,\tau}\ =\ a_\sigma\cdot X_{w,t}\,,
\]
 where $\sigma$ is a minimal face of the face containing $w$ in its relative interior.

 If $\rho\in\scrC$, then the elements 
 $\{a_\sigma\mid \sigma\subset\rho\mbox{ and }\sigma\in\scrM(\scrC)\}$ lie in a single coset of
 $\T_{\langle\rho\rangle}$. 
\end{proposition}

%
\subsection{Complex coamoebae}

As a real algebraic group, the set $\C^*$ of invertible complex numbers is
isomorphic to  $\R\times\U$ under the map $(r,\theta)\mapsto e^{r+\sqrt{-1}\theta}$.
Here, $\U$ is set of complex numbers of norm 1 which we identify with $\R/2\pi\Z$.
The inverse map $z\mapsto (\log|z|,\arg(z))$ induces an isomorphism
$\T_N\to \R_N\times \U_N$, 
where $\R_N:=\R\otimes_\Z N$ and $\U_N:=\U\otimes_\Z N$.
Composing with the projections to the factors $\R_N$ and $\U_N$ gives the maps
$\Log$ and $\Arg$, respectively.

The \demph{amoeba  $\scrA(X)$} of a subscheme $X\subset\T_N$ is the image of its complex
points $X(\C)$ under the map $\Log$.
Let $\bS_N$ be the sphere in $\R_N$, with projection $\pi\colon\R_N\setminus\{0\}\to\bS_N$.
Then the \demph{logarithmic limit set $\scrL^\infty(X)$} of $X$ is the set of accumulation
points in $\bS_N$ of sequences $\{\pi(\Log(x_n))\}$ where 
$\{x_n\}\subset X$ is unbounded in that its sequence of
logarithms $\{\Log(x_n)\}$ is unbounded.
This was introduced by Bergman~\cite{Berg}, and then Bieri and Groves~\cite{BiGr} showed
that the cone over it was a rational polyhedral fan.
Later work of Kapranov~\cite{EKL}, and then  Speyer and Sturmfels~\cite{SS} identified this
fan with the negative $-\trop(X)$ of the tropical variety of $X$, computed in any valued
field $\K$ containing $\C$ with value group $\R$ and residue field $\C$.

\begin{example}\label{Ex:AlineP2}
 Let $\ell\subset\T^2$ be defined by $f:=x+y+1=0$.
 By the triangle inequality,
\[
   |x|+|y|\ \geq\ 1 \,,\quad
   |x|+1\ \geq\ |y|\,,\quad\mbox{and}\quad
   |y|+1\ \geq\ |x|\,.
\]
 Here is the solution of these inequalities, the  amoeba
 $\scrA(\ell)$, and the tropical variety $\trop(\ell)$.
\[
   \raisebox{-30pt}{%
   \begin{picture}(94,80)(-14,0)
     \put(0,0){\includegraphics[height=80pt]{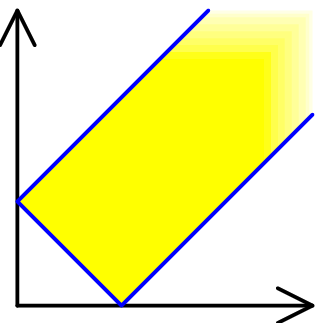}}
     \put(-14,70){$|y|$}  \put(66,14){$|x|$}
   \end{picture}
    \qquad\quad
   \includegraphics[height=80pt]{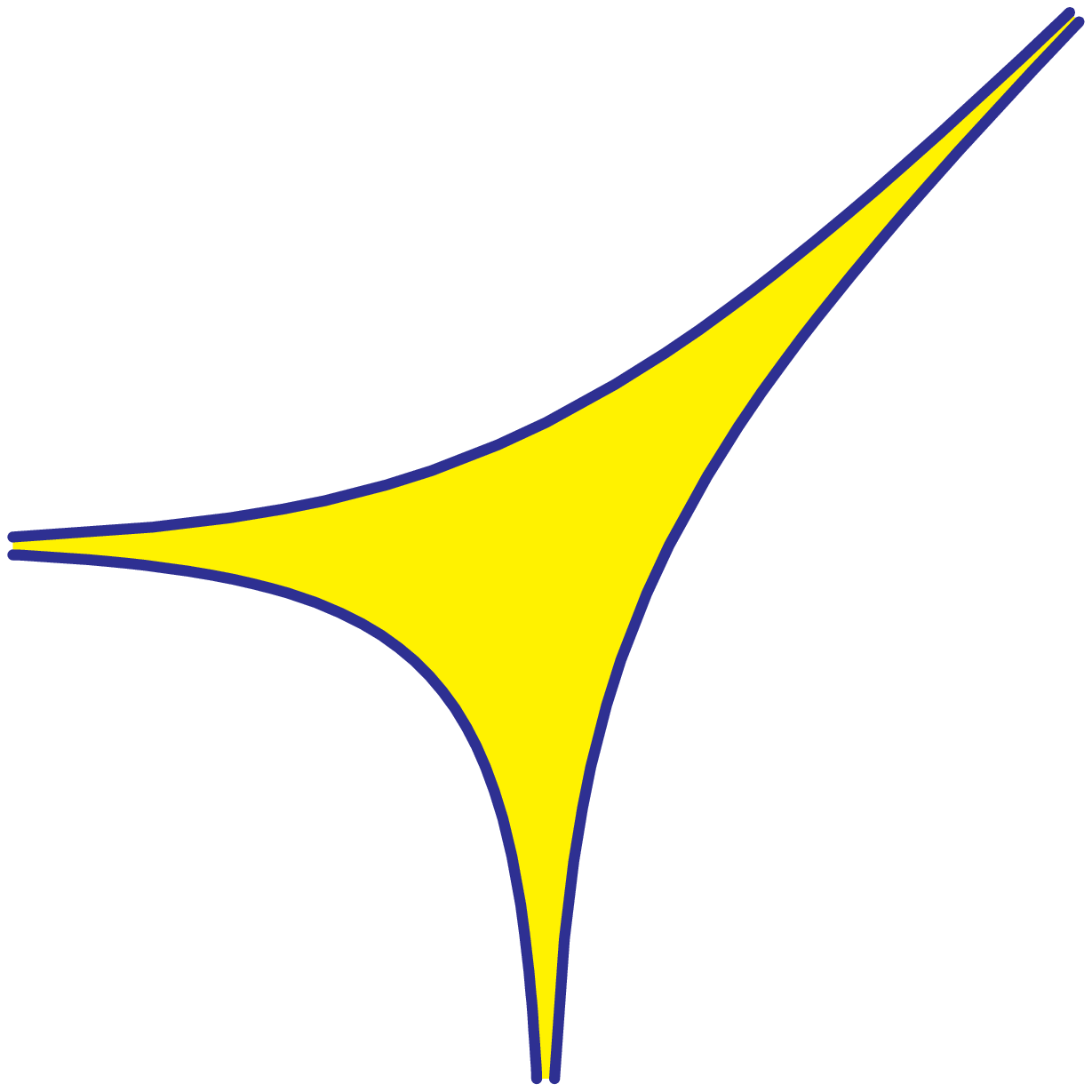}
    \qquad\quad
   \begin{picture}(80,80)
     \put(0,0){\includegraphics[height=80pt]{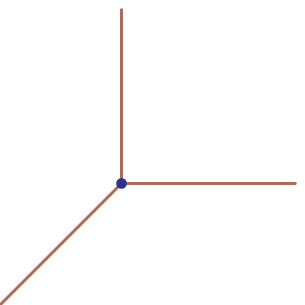}}
     \put(2,34){$(0,0)$}  
   \end{picture}}
  \eqno{\QED}
\]
\end{example}

The \demph{coamoeba $\coscrA(X)$} of a complex subscheme $X$ of $\T_N$ is the image of
$X(\C)$ under the argument map $\Arg$. 
Coamoebae are naturally semi-algebraic sets.
Its \demph{phase limit set $\pls(X)$} is the set of accumulation
points of arguments of unbounded sequences in $X$.
For $w\in\Z^n$, the initial scheme $\ini_w X\subset(\C^*)^n$ is defined by the initial
ideal of $I$. 
Then $\ini_w X\neq \emptyset$ exactly when $w$ lies in $\trop(X)$.
For the line $\ell$ of Example~\ref{Ex:AlineP2}, $\ini_{(0,0)}f=f$, and for $s>0$,
\[
   \ini_{(s,0)}f\ =\ y+1\,,\qquad
   \ini_{(0,s)}f\ =\ x+1\,,\qquad\mbox{and}\qquad
   \ini_{(-s,-s)}f\ =\ x+y\,,
\]
and if $w$ does not lie in one of these rays (which constitute $\trop(\ell)$), then $\ini_wf$
is a monomial, and so $\ini_w\ell=\emptyset$. 

The following structure theorem for the phase limit set was established in~\cite{NS}.
The case when $X$ is a complete intersection earlier shown by
Johansson~\cite{Johansson}. 

\begin{proposition}\label{T:one}
  The closure of $\coscrA(X)$ is $\coscrA(X)\cup \pls(X)$, and 
\[
   \pls(X)\ =\  \bigcup_{w\neq 0} \coscrA(\ini_w X)\,.
\]
\end{proposition}

This union is finite, for $X$ has only finitely many initial schemes.

\begin{example}\label{Ex:ClineP2}
 We determine the coamoeba of the line $\ell$ from Example~\ref{Ex:AlineP2}. 
 If $x$ is real then so is $y$.
 As $\ell$ meets all but the positive quadrant, its real points give 
 $(\pi,0)$, $(\pi,\pi)$, and $(0,\pi)$ in $\coscrA(\ell)$ (the dots
 in~\eqref{Eq:two_triangles}). 
 For $x\not\in\R$, as $y=-(x+1)$, the diagram
 \[
   \begin{picture}(180,53)(-40,-1)
    \put(0,0){\includegraphics{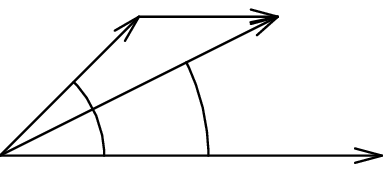}}
    \put(18,43){$x$}
    \put(82,43){$x+1$}
    \put(115,2){$\R$}    \put(-9,-2){$0$}
    \put(84,20){$\arg(x+1)$} \put(82,22.5){\vector(-4,-1){20}}
    \put(-40,16){$\arg(x)$}\put(-4,18.5){\vector(4,-1){30}}
   \end{picture}
 \]
 shows that $\arg(y)=\pi+\arg(x+1)$ can take on any value
 strictly between $\pi+\arg(x)$ (for $x$ near $\infty$) and $0$  (for $x$ near $0$).
 Thus $\coscrA(\ell)$ also contains the interiors of the two triangles
 of~\eqref{Eq:two_triangles} displayed 
 in the fundamental domain $[-\pi,\pi]^2$ in $\R^2$ of $\U^2$.
 (This is modulo $2\pi$, so that $\pi=-\pi$.) 
 Because of this identification, the triangles meet at their vertices, and
 each shares half of one of the lines $\arg(x)=\pi$, $\arg(y)=\pi$, and 
 $\arg(y)=\pi+\arg(x)$.
  \begin{equation}\label{Eq:two_triangles}
   \raisebox{-50pt}{%
    \begin{picture}(185,88)(-25,-12)
     \put(-2,-2){\includegraphics{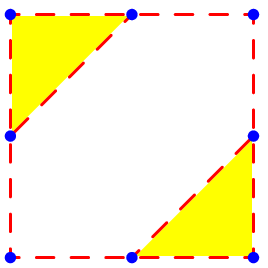}}
     \put(-25,-2){$-\pi$} \put(-15,33){$0$} \put(-15,69){$\pi$}
     \put(-10,-12){$-\pi$} \put(33,-12){$0$} \put(68,-12){$\pi$}
     \put(90,2){\vector(1,0){70}} \put(106,8){$\arg(x)$}
     \put(83,15){\vector(0,1){46}} \put(87,33){$\arg(y)$}
   \end{picture}}
  \end{equation}
 The phase limit set $\pls(\ell)$ consists of the three lines bounding the
 triangles, which are the coamoebae of the three non-trivial initial schemes of $\ell$.
\qed
\end{example}

If $X=\T_{N'}\subset\T_N$ is a subtorus, then $\coscrA(X)\subset\U_N$ is the corresponding real 
subtorus, $\U_{N'}$.
More generally, if $\T_{N'}$ acts on a scheme $X$ with finitely many orbits (so that $X$ is
supported on a union of translates of $\T_{N'}$) then $\coscrA(X)$ is a finite union of
translates of $\U_{N'}$, and its dimension is equal to the rank of $N'$, which is the complex
dimension of $X$.

%
\section{Non-archimedean coamoebae and phase tropical varieties}\label{S:main}

Suppose now that $\K$ is a valued field with surjective valuation map
$\nu\colon\K^*\twoheadrightarrow\Gamma$ whose residue field is the complex
numbers.
Choose a section $\tau\colon\Gamma\to\K^*$ of the valuation.
With this choice, we define an argument map $\arg_\tau\colon\K^*\to\U$ 
as follows.
For each element $x\in\K^*$, there is a unique element $a_x\in R^*$ such that
$x=a_x\cdot \tau(\nu(x))$, and we set $\defcolor{\arg_\tau}(x):= \arg(\overline{a_x})$, the
argument of the reduction of $a_x$.
This induces an argument map $\Arg_\tau\colon\T_N(\K)\to\U_N$, which 
is also the composition
\[
   \T_N(\K)\ \to\ \coprod_{w\in\Gamma_N} \T^w(\C)\ 
    \xrightarrow{\ \varphi^{-1}_\tau\ }\ \Gamma_N\times\T_N(\C)\ 
    \twoheadrightarrow\ \T_N(\C)\ \xrightarrow{\ \arg\ }\ \U_N\,.
\]
We assume that $\tau$ is fixed and drop the dependence on $\tau$ from our notation.

%
\subsection{Non-archimedean coamoebae}

The \demph{non-archimedean coamoeba $\nca(X)$} of a subscheme $X$ of $\T_N$ is the image
of $X(\K)$ under the map $\Arg$.
It is therefore the union of the complex coamoebae of all tropical reductions of $X$.
This is in fact a finite union.
A projective space closure $\P^N$ of $\T_N$ gives a polyhedral complex $\scrC$ in $\Gamma_N$
with support $\trop(X)$ and by Proposition~\ref{Prop:ini}, we have
\[
   \nca(X)\ =\ \bigcup_{\sigma\in\scrC} \coscrA(X_\sigma)\,.
\]

\begin{example}\label{Ex:tropical_plane_line}
  Let $1\in\Gamma$ be a  positive element, and write $t$ for
  $\tau^1\in\K^*$. 
  Suppose that $\ell$ is the line defined in $(\K^*)^2$ by $f:=x+y+t=0$.
  Its tropical variety is shown on the left in Figure~\ref{F:ellInPlane}.
 (We assume that $\Gamma\subset\R$.)
  We have $\ini_{(1,1)}f=x+y+1$, and for $s>0$,
\[
   \ini_{(1+s,1)}f\ =\ y+1\,,\qquad
   \ini_{(1,1+s)}f\ =\ x+1\,,\qquad\mbox{and}\qquad
   \ini_{(1-s,1-s)}f\ =\ x+y\,,
\]
  so $\nca(\ell)$ is the union of the coamoebae of
  these tropical reductions, which is the union of the two closed triangles on the 
  right in Figure~\ref{F:ellInPlane}.
\begin{figure}[htb]
  \begin{picture}(135,96)(-30,-10)
   \put(0,0){\includegraphics{figures/tropical_plane_line.eps}}
   \put(2,35){$(1,1)$}
   \put(-30,-10){$(1{-}s,1{-}s)$}
   \put(38,75){$(1,1{+}s)$}
   \put(65,41){$(1{+}s,s)$}
  \end{picture}
  \qquad\qquad
   \begin{picture}(102,96)(-20,-17)
     \put(-2,-2){\includegraphics{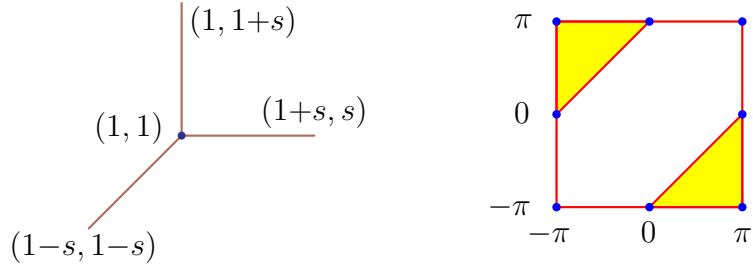}}
     \put(-25,-2){$-\pi$} \put(-15,33){$0$} \put(-15,69){$\pi$}
     \put(-10,-12){$-\pi$} \put(33,-12){$0$} \put(68,-12){$\pi$}
   \end{picture}
\caption{Tropical line and its non-archimedean coamoeba}\label{F:ellInPlane}
\end{figure}
A different section of the valuation map will simply translate this coamoeba, as the tropical
variety has a unique minimal face.
\qed
\end{example}

\begin{theorem}\label{T:complex_case}
 If $X\subset\T_N$ is defined over $\C$, then the closure of its complex coamoeba is equal
 to its non-archimedean coamoeba.
\end{theorem}

Note that in particular, $\nca(X)$ is closed.

\begin{proof}
 By definition, $\nca(X)$ is the union of the coamoebae of all tropical reductions of $X$.
 Since $X$ is a complex scheme, these are its initial schemes, so the result
 holds by Proposition~\ref{T:one}.
\end{proof}

We will see that Theorem~\ref{T:complex_case} describes the local structure of non-archimedean
coamoebae. 

\begin{example}\label{Ex:tropical_space_line}
 Let $\zeta$ be a primitive third root of unity, $\omega:=1+\sqrt{-1}$,
 and $t$ be an element of $\K^*$ with valuation 1 as in
 Example~\ref{Ex:tropical_plane_line}. 
 Consider the line 
\[
   x+\zeta y + \zeta^2 t\ =\ 
   \sqrt{-1}\cdot x +z - \omega\ =\ 0
\]
 in $(\K^*)^3$.
 We display its tropical variety, again assuming that $\Gamma\subset\R$.
\[
  \begin{picture}(255,95)(-66,-1)
   \put(0,0){\includegraphics{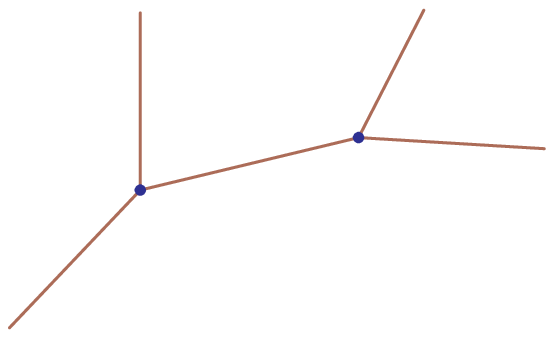}}
   \put(-6,40){$(0,0,0)$}    \put(82,40){$(1,1,0)$} 
   \put(-1,85){$(0,0,s)$}     \put(122,85){$(1,1{+}s,0)$}
   \put(-66,0){$(-s,-s,-s)$} \put(140,40){$(1{+}s,1,0)$}
  \end{picture}
\]
This has two vertices $(0,0,0)$ and $(1,1,0)$ connected by an internal line segment
($(s,s,0)$ for $0<s<1$) and four rays as indicated for $s>0$.
For each face of this polyhedral complex, the ideal $I$ of the line has a different initial
ideal as follows 
\[
  \begin{array}{rclrcl}
   \ini_{(0,0,0)}I&=&\langle x+\zeta y,\, \sqrt{-1}\cdot x +z -\omega\rangle
    \rule{0pt}{14pt}&
   \ini_{(1,1,0)}I& =& \langle x+\zeta y+\zeta^2,\, z - \omega\rangle\rule{0pt}{14pt} 
   \\
    \ini_{(s,s,0)}I& =&\langle x+\zeta y,\,z - \omega)\rangle&
   \ini_{(1+s,1,0)}I&=&\langle \zeta y+\zeta^2,\, z - \omega\rangle
     \\\rule{0pt}{14pt}
   \ini_{(0,0,s)}I&=&\langle x+\zeta y,\, \sqrt{-1}\cdot x - \omega\rangle
    &
   \ini_{(1,1+s,0)}I&=&\langle x+\zeta^2,\, z - \omega\rangle
    \\\multicolumn{6}{c}{\rule{0pt}{14pt}
   \ini_{(-s,-s,-s)}I\ =\ \langle x+\zeta y,\, \sqrt{-1}\cdot x + z\rangle}
   \end{array}
\]

Figure~\ref{F:NA} shows two views of $\nca(\ell)$, which is the union of 
the seven coamoebae, one for each tropical reduction of $\ell$.
\begin{figure}[htb]
  \includegraphics[height=150pt]{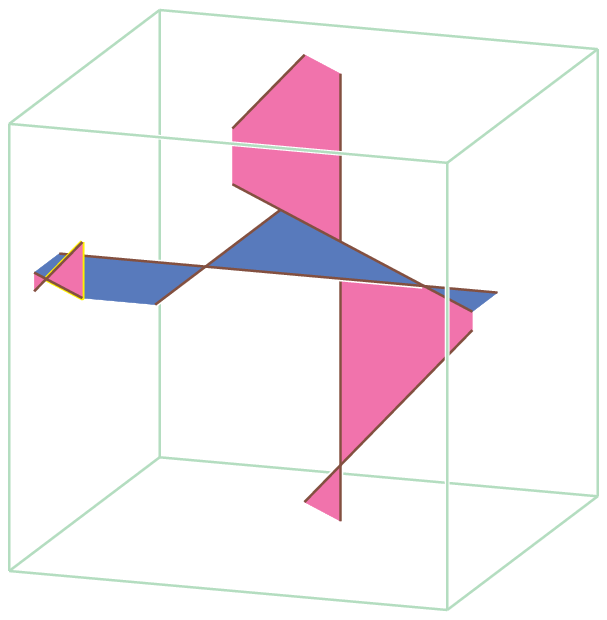}   \qquad
  \includegraphics[height=165pt]{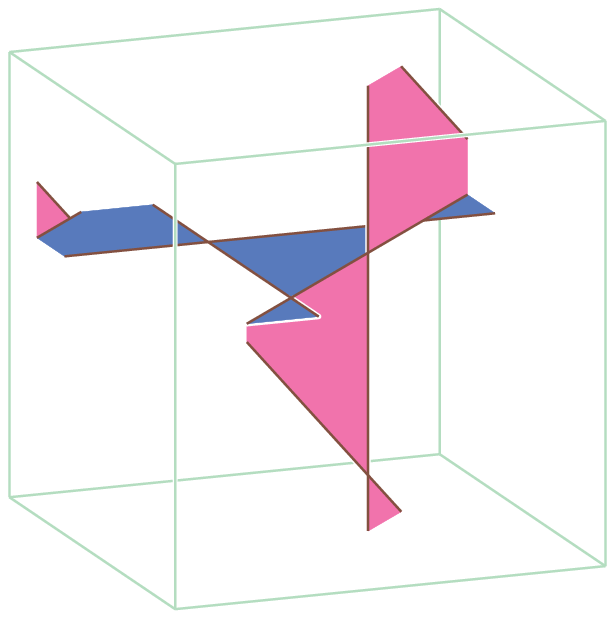} 
\caption{Non-archimedean coamoeba of a line}
\label{F:NA}
\end{figure}
It is the closure of the two coamoebae corresponding to the vertices of $\trop(\ell)$, each of
which is a coamoeba of a line in a plane consisting of two triangles.
Here $(0,0,0)$ corresponds to the vertical triangles and $(1,1,0)$ to the 
horizontal triangles.
These two coamoebae are attached along the coamoeba of the tropical reduction corresponding
to the edge between the vertices, and each has two boundary coamoebae corresponding to the
unbounded rays at each vertex.

If we change the section of the valuation which we implicitly used to compute
this non-archimedean coamoeba, its effect would be to rotate (slide) the horizontal and vertical
triangles along the line they share, followed by a global translation.
\qed
\end{example}

\begin{theorem}\label{T:nca}
 Let $X$ be an irreducible $\K$-subscheme of\/ $\T_N$, where $\K$ is a field with a
 non-archimedean valuation and residue field $\C$.
 Fix a section of the valuation to define the non-archimedean coamoeba
 of $X$.
 Let $\scrC$ be a polyhedral complex with support $\trop(X)$ induced by a projective
 compactification of\/ $\T_N$.
 Then
 \begin{equation}\label{E:NCA}
   \nca(X)\ =\ \bigcup_{\substack{\sigma\in\scrC\\\mbox{is minimal}}}
      \overline{\coscrA(X_\sigma)}
      \ =\ \bigcup_{\substack{\sigma\in\scrC\\\mbox{is minimal}}}
       \nca(X_\sigma)\,,
 \end{equation}
 the union of closures of coamoebae of the tropical reductions 
 for all minimal faces of $\scrC$.

 Changing the section of valuation will translate each piece $\nca(X_\sigma)$ by a possibly
 different element of\/ $\U_N$.
 If $\rho\in\scrC$ is not minimal, then the translations for the minimal faces of $\rho$ all
 lie in the same coset of\/ $\U_{\langle\sigma\rangle}$.
\end{theorem}

\begin{proof}
 Let $\rho\in\scrC$ be a face.
 Then $\rho$ contains a minimal face $\sigma$ of $\scrC$.
 Let $w\in\sigma$.
 By Proposition~\ref{Prop:local}, $X_\rho$ is an initial scheme of $X_w=X_\sigma$.
 By by Proposition~\ref{T:one}, its coamoeba is a component of the phase limit set of
 $X_\sigma$ and lies in the closure of $\coscrA(X_w)$.
 Similarly, Theorem~\ref{T:complex_case} implies that $\coscrA(X_\rho)\subset\nca(X_\sigma)$.
 These facts imply~\eqref{E:NCA}.
 The statement about the change of section follows from
 Proposition~\ref{P:trivialization}. 
\end{proof}

%
\subsection{Phase tropical varieties}

The product of the maps $\trop$ and $\Arg$ is a map 
\[
  \defcolor{\Ptrop}\ \colon\ \T_N(\K)\ \longrightarrow\ \Gamma_N\times\U_N\,.
\]
The \demph{phase tropical variety $\Ptrop(X)$} of a $\K$-subscheme $X$ of $\T_N$ is the 
closure of the image of $X(\K)$ under this map.
Here, we equip $\Gamma_N$ with the weakest topology so that the
half-spaces~\eqref{Eq:half-space} are closed.
Under the  projection to the first coordinate, $\Ptrop(X)$ maps to the tropical variety of
$X$ with fiber over a point $w$ the closure of the coamoeba of the tropical reduction
$\ini_wX$ of $X$. 
This notion was introduced by Mikhalkin for plane curves~\cite[\S6]{Mik05}.

\begin{example}
 Suppose that $\Gamma=\R$ and consider $\Ptrop(\ell)$, where $\ell$ is the line $x+y+1=0$.
 As we cannot easily represent objects in four dimensions, we draw $\trop(\ell)$ and indicate
 the fibers above points of $\trop(\ell)$ via screens.
\[
   \includegraphics{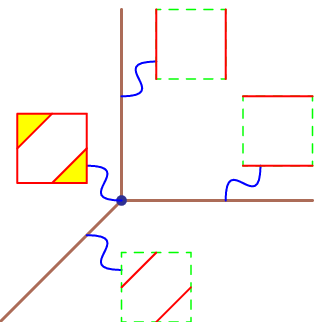}
\]
 This phase tropical variety is homeomorphic to a three-punctured sphere (in fact to $\ell$
 itself as a subset of $(\C^*)^2$).
 This indicates why we take the closure, for without the closure, the phase tropical variety
 would omit the six line segments that are in the closure of the coamoeba of
 $\ini_{(0,0)}\ell$ above $(0,0)$, but not in the coamoeba.\qed
\end{example}

Choose a projective compactification $\P^N$ of $\T_N$ so that for any $\K$-subscheme $X$ of
$\T_N$, we have a polyhedral complex $\scrC$ whose support is $\trop(X)$, as in
Remark~\ref{R:scrC}. 
A subset of the form $\rho\times Y$ of $\Gamma_N\times \U_N$, where $\rho$ is a polyhedron (or
the relative interior of a polyhedron) and $Y$ is a semi-algebraic subset of $\U_N$, has 
dimension  $\dim(\rho)+\dim(Y)$, where $\dim(\rho)$ is the dimension of the polyhedron
as defined in Subsection~\ref{S:polyhedron} and $\dim(Y)$ is its dimension as a semialgebraic
subset of $\U_N$.
Any finite union of such sets has dimension the maximum dimension of the sets in the
union.

\begin{theorem}\label{T:phaseTropical}
 Let $X$ be a $\K$-subscheme of $\T_N$.
 Then the phase tropical variety is a subset of $\Gamma_N\times \U_N$ of dimension $2\dim(X)$.
 It is the (finite) union 
 \begin{equation}\label{Eq:PTV}
   \Ptrop(X)\ =\  \bigcup_{\sigma\in\scrC} \sigma\times\overline{\coscrA(X_\sigma)}
   \ =\ 
   \coprod_{\sigma\in\scrC} \relint{\sigma}\times\overline{\coscrA(X_\sigma)}\,.
 \end{equation}
 If we change the section of valuation, the coamoeba factors in this
 decomposition will be translated by elements $a_\sigma\in\U_N$.
 For any face $\sigma$ of $\scrC$ the elements $a_\rho$ corresponding to subfaces $\rho$ of
 $\sigma$ all lie in a single coset of\/ $\U_{\langle\sigma\rangle}$.
\end{theorem}

\begin{proof}
 The decomposition~\eqref{Eq:PTV} is immediate from the definition of phase tropical varieties.
 We have the formula for the dimension of a coamoeba of a $\K$-subscheme $X$ of $\T_N(\C)$,
\[
   \dim(\coscrA(X))\ \leq\ \min(\rank(N), 2\dim(X)-n)\,,
\]
 when a subtorus $\T$ of $\T_N(\C)$ of dimension $n$ acts on $X$.
 This is because $\T\simeq\U^n\times\R_>^n$, the argument map has fibers containing orbits
 of $\R_>^n$, and the image must have dimension at most $\dim(\U_N)=\rank(N)$. 
 
 Suppose that $X$ is irreducible.
 For a face $\sigma$ of $\scrC$, $X_\sigma$ has an action of the subtorus 
 $\T_{\langle\sigma\rangle}$, which has dimension $\dim(\sigma)$.
 Thus $\dim(\sigma\times\overline{\coscrA(X_\sigma)})$ is
\[
  \dim(\sigma)\ +\ \dim(\coscrA(X_\sigma))\ \leq\ 
  \dim(\sigma)\ +\ \min(\rank(N), 2\dim(X_\sigma)-\dim(\sigma))\,,
\]
 which is at most $2\dim(X)$, 
 as $\dim(X)=\dim(X_\sigma)$.
 If $\sigma$ is a face of maximal dimension, so that $\dim(\sigma)=\dim(X)(\leq\rank(N))$, 
 then $\dim(X_\sigma)=\dim(\sigma)$, and so 
 $\dim(\sigma\times\overline{\coscrA(X_\sigma)})=2\dim(\sigma)=2\dim(X)$.

 The last statement follows from Proposition~\ref{P:trivialization}.
\end{proof}

%
\section{Coamoebae of tropically simple varieties}\label{S:simple}

Unlike tropical varieties, non-archimedean coamoebae and phase tropical varieties are 
typically not unions of polyhedra.
This is because any complex coamoeba may occur as a component of a non-archimedean 
coamoeba, and complex coamoebae are only partially combinatorial, as we may see from the
coamoeba of the hyperbola $(1+\sqrt{-1})-x-y+xy=0$,
\[
   \includegraphics{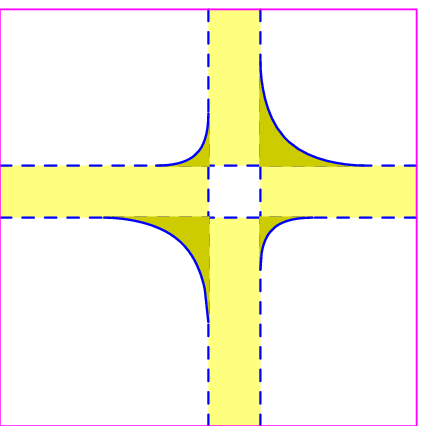}
\]
The coamoeba mostly lies inside the two rectangles enclosed by its phase limit set, which 
consists of the four lines where the arguments of $x$ and $y$ are $0$ and $\pi/4$.
It also has some extra pieces with curvilinear boundary that are shaded in this picture.

A $\K$-subscheme $X$ of $\T_N$ is \demph{tropically simple} if every tropical reduction 
$X_w$ is defined by polynomials, each of which has affinely independent 
support, and the spans of these supports are themselves independent.
We give a less terse and more geometric description in Subsections~\ref{SS:simple}
and~\ref{S:TSV}.
Non-archimedean coamoebae and phase tropical varieties of tropically simple varieties are
unions of polyhedra. 
We first study some complex varieties whose coamoebae are unions of polyhedra, for they form
the local theory of tropically simple varieties. 

\subsection{Coamoebae of hyperplanes}

A hypersurface $H\subset\T_N(\C)$ whose Newton polytope is a
unimodular simplex is a \demph{hyperplane}.
There is an identification of $T_N(\C)$  with $(\C^*)^{\rank(N)}$ 
and non-zero
complex numbers $a_1,\dotsc,a_n$ ($n\leq \rank(N)$) such that $H$ is defined by 
 \begin{equation}\label{Eq:hyperplane}
    1\ +\ a_1 z_1\ +\ a_2 z_2\ +\ \dotsb\ + a_n z_n\ =\ 0\,,
 \end{equation}
where $z_1,z_2,\dotsc$ are the coordinate functions on $(\C^*)^{\rank(N)}$.
Then $H$ is the product of the hyperplane in $\T^n(\C)$ with
equation~\eqref{Eq:hyperplane} and a coordinate subtorus of dimension $\rank(N)-n$. 
The coamoeba of $H$ is the product of the coamoeba of~\eqref{Eq:hyperplane} and a
coordinate subtorus $\U^{\rank(N)-n}$.

We may change coordinates in~\eqref{Eq:hyperplane}, replacing $a_i z_i$ by $x_i$ to obtain
the equation
 \begin{equation}\label{Eq:hyperplane_simple}
  \defcolor{f}\ :=\  1\ +\ x_1\ +\ x_2\ +\ \dotsb\ + x_n\ =\ 0\,.
 \end{equation}
This simply translates the coamoeba by the vector
\[
   (\arg(a_1)\,,\,\arg(a_2)\,,\,\dots\,,\,\arg(a_n))\,.
\]

A hyperplane $H$ in $(\C^*)^n$ with equation~\eqref{Eq:hyperplane_simple}
is \demph{standard}.
To describe its coamoeba, we work in a fundamental domain of homogeneous
coordinates for $\U^n$, namely 
\[
   \bigl\{(\theta_0,\theta_1,\dotsc,\theta_n)\ :\ \theta_i\leq
    2\pi+\min\{\theta_0,\dotsc,\theta_n\}\bigr\}\,,
\]
 where we identify 
$(\theta_0,\dotsc,\theta_n)$ with its translate $(\theta_0+t,\dotsc,\theta_n+t)$ for any
$t$. 

\begin{lemma}\label{L:hyperplane}
 Let $H$ be a standard hyperplane~$\eqref{Eq:hyperplane_simple}$. 
 Then $\overline{\coscrA(H)}$ is the 
 complement of 
 \begin{equation}\label{Eq:Zonotope}
   \defcolor{Z}\ :=\ 
    \{(\theta_0,\theta_1,\dotsc,\theta_n): |\theta_i-\theta_j|<\pi\ 
     \mbox{ for all }\ 0\leq i<j\leq n\}\,.
 \end{equation}
\end{lemma}

\begin{proof}
 The set $Z$ is an open zonotope.
 We first show that no point of $Z$ lies in the coamoeba.
 Let $(\theta_0,\dotsc,\theta_n)\in Z$.
 The inequalities defining $Z$ are invariant under permuting coordinates and
 simultaneous translation.
 Thus we may assume that 
 \begin{equation}\label{Eq:ordered}
   0=\theta_0\ \leq\ \theta_1\ \leq\ \dotsb\ \leq\ \theta_n\ <\ \pi\,.
 \end{equation}
 If $1=x_0,x_1,\dotsc,x_n$ are complex numbers with $\arg(x_i)=\theta_i$ for all
 $i$, then they all have nonnegative imaginary parts.
 They cannot satisfy~\eqref{Eq:hyperplane_simple} for either all are positive real numbers or
 one has a non-zero imaginary part.
 Thus no point of $Z$ lies in the coamoeba of $H$.

 We use the phase limit set for the other inclusion.
 Let $\coscrA(H)$ be the coamoeba of $H$.
 Suppose that $\theta=(\theta_0,\theta_1,\dotsc,\theta_n)$ lies in the complement of $Z$,
 so that $|\theta_i-\theta_j|\geq \pi$ for some pair $i,j$.
 We may assume that $i=1$, $j=2$, and $\theta_0=0$.
 Then $(\theta_1,\theta_2)$ lies in the set
 \begin{equation}\label{Eq:line_coamoeba}
   \raisebox{-40pt}{\begin{picture}(102,88)(-20,-12)
     \put(-2,-2){\includegraphics{figures/closed_triangles.eps}}
     \put(-25,-2){$-\pi$} \put(-15,33){$0$} \put(-15,67){$\pi$}
     \put(-10,-12){$-\pi$} \put(32,-12){$0$} \put(66,-12){$\pi$}
   \end{picture}}
 \end{equation}
 which is the closure of the coamoeba of the line given by $1+x_1+x_2=0$.
 This is an initial form $\ini_w f$ of the polynomial~\eqref{Eq:hyperplane_simple}, and so
 by Proposition~\ref{T:one}, $\overline{\coscrA(H)}$ contains $\overline{\coscrA(\ini_w H)}$,
 which is the product of the coamoeba~\eqref{Eq:line_coamoeba} and the coordinate torus spanned
 by the last $n-2$ coordinates.
 This shows that $\theta\in\overline{\coscrA(\ini_w H)}\subset \overline{\coscrA(H)}$, which
 completes the proof.
\end{proof}

This proof also shows that these portions of the phase limit set, which are cylinders with base
the coamoeba of a line in a plane, cover the coamoeba of $H$.
These correspond to two-dimensional (triangular) faces of the Newton polytope of $f$.

\begin{example}\label{Ex:hyperplane}
 Here are three views of the closure of the coamoeba of the plane
 $1+x_1+x_2+x_3=0$ in $(\C^*)^3$, in the fundamental domain $[-\pi,\pi]^3$ for 
 $\U^3$.
 \[
   \includegraphics[height=3.5cm]{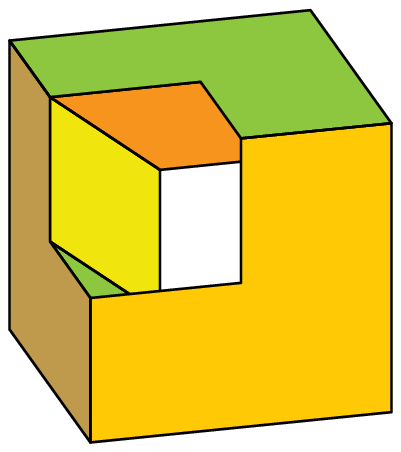} \quad\qquad
   \includegraphics[height=3.5cm]{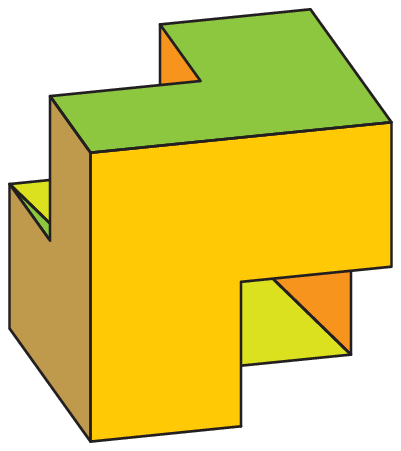} \quad\qquad  
   \includegraphics[height=3.5cm]{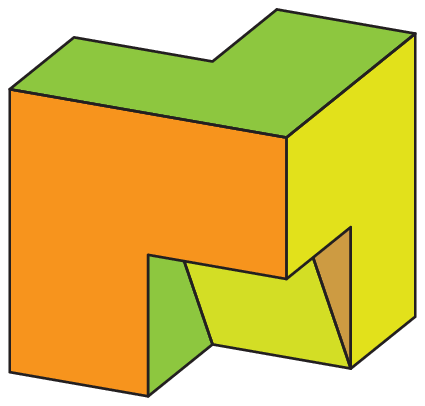} 
 \]
 In this domain, the zonotope consists of those $(\theta_1,\theta_2,\theta_3)$ such that all six 
 quantities $\theta_1$, $\theta_2$, $\theta_3$, $\theta_1-\theta_2$, $\theta_2-\theta_3$, and
 $\theta_3-\theta_1$ lie between $-\pi$ and $\pi$.
 Besides the 12 boundary quadrilaterals lying on the six coordinate planes defined by these
 inequalities, the phase limit set has four components which are cylinders over coamoebae of
 planar lines.
 \[
   \includegraphics[height=3.3cm]{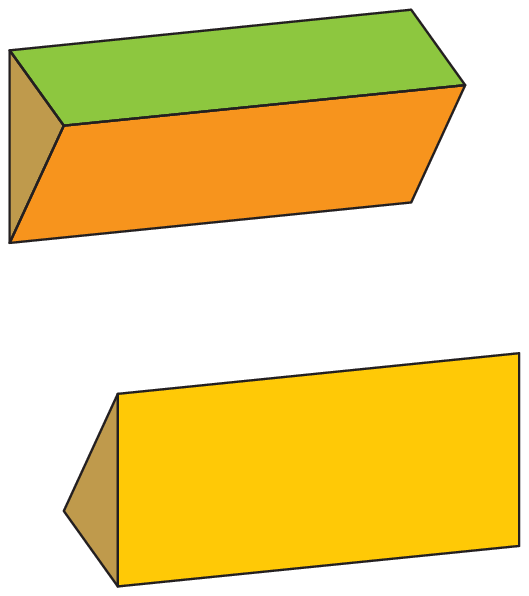} \quad\qquad
   \includegraphics[height=3.3cm]{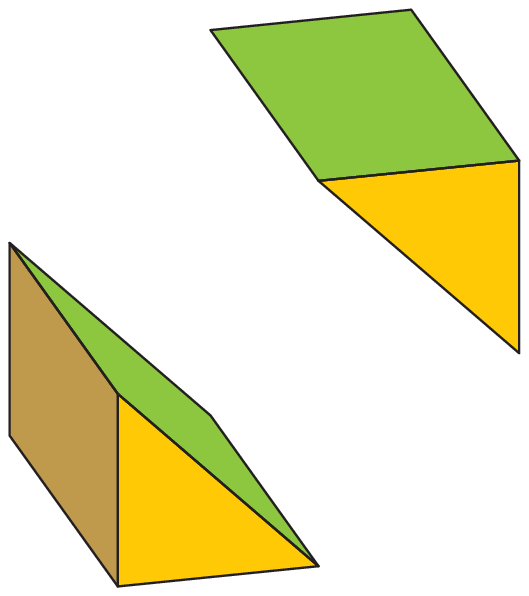} \quad\qquad
   \includegraphics[height=3.3cm]{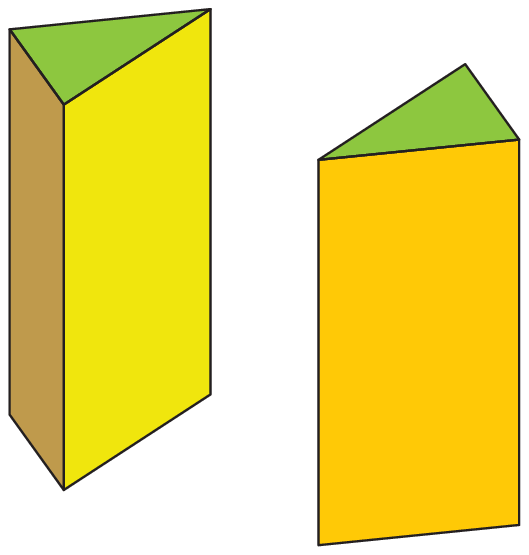} \quad\qquad
   \includegraphics[height=3.5cm]{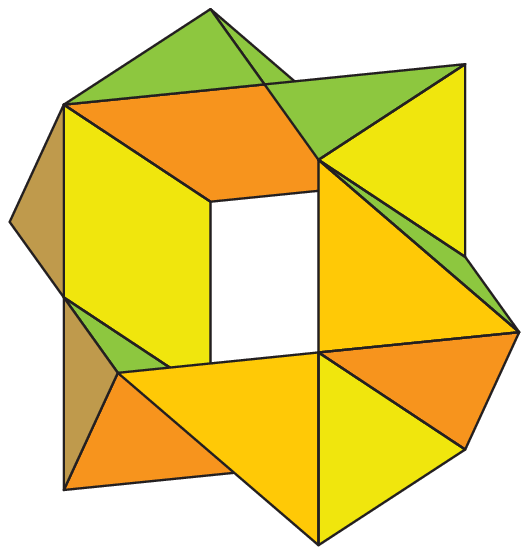} 
 \]
 The closures of these cylinders cover the coamoeba of $H$, as we showed in
 Lemma~\ref{L:hyperplane}. 
 Each of these cylinders is itself the coamoeba of a hyperplane.
 \qed
\end{example}

\begin{theorem}\label{T:hyperplane}
 Let $H\subset\T_N(\C)$ be a hyperplane whose defining polynomial $f$ has Newton polytope
 the simplex $\Delta$.
 If $\dim(\Delta)=1$, then $\coscrA(H)$ is a translate of a subtorus 
 of dimension $\rank(N){-}1$.
 Otherwise, $\coscrA(H)$ has dimension $\rank(N)$ and is a cylinder over the coamoeba of a 
 hyperplane in $(\C^*)^{\dim(\Delta)}$, which is the complement of a
 zonotope.

 The phase limit set of $H$ is the union of coamoebae of hyperplanes defined by the
 restrictions of $f$ to the non-vertex faces of $\Delta$.
 Those corresponding to triangles cover $\coscrA(H)$.

 The non-archimedean coamoeba of $H_\K\subset\T_N(\K)$ is the closure of the coamoeba of
 $H$ and is the complement in $\U_N$ of cylinder over an open zonotope of
 dimension $\dim(\Delta)$.

 The tropical variety $\trop(H)$ of $H_\K$ is the normal fan of $\Delta$ with its
 full-dimensional cones removed.
 Tropical reductions $H_\sigma$ for $\sigma$ a cone of $\trop(H)$ are the hyperplanes
 defined by the restriction of $f$ to the face of $\Delta$ dual to $\sigma$.
 The phase tropical variety $\Ptrop(H)$ of $H$ is the closure of its restriction to the
 cones of $\trop(H)$ corresponding to edges and two-dimensional faces of $\Delta$ and it
 is a pure-dimensional polyhedral complex of dimension $2\dim(H)$.
\end{theorem}

\begin{proof}
 These statements are either immediate from the definitions or from
 Lemma~\ref{L:hyperplane}, or are well-known (e.g.\ the description of 
 $\trop(H)$), except perhaps those concerning $\Ptrop(H)$.

 By Theorem~\ref{T:phaseTropical}, $\dim\Ptrop(H)=2\dim(H)$ with the pieces of the
 decomposition~\eqref{Eq:PTV} corresponding to maximal cones $\sigma$ of $\trop(X)$ having
 maximal dimension. 
 Let $\sigma$ be a non-maximal cone dual to a face $\calF$ of $\Delta$ 
 of dimension two or greater.
 Then $\dim(\sigma)=\rank(N)-\dim(\calF)$, and the coamoeba of the tropical reduction
 $H_\sigma$ has dimension $\rank(N)$. 
 Thus the piece of $\Ptrop(X)$ corresponding to $\sigma$ has dimension 
 $2\rank(N)-\dim(\calF)$.
 This is maximized with maximal value $2\rank(N)-2=2\dim(H)$ when $\calF$ a triangular face of
 $\Delta$.

 \demph{Ridges} are faces of $\trop(X)$ corresponding to triangular faces of $\Delta$.
 Every non-maximal face of $\trop(X)$ is contained in a ridge, and 
 by the computation in the proof of Lemma~\ref{L:hyperplane} the coamoeba of the tropical 
 reduction corresponding to a non-maximal face $\sigma$ is the union of the coamoebae of
 tropical reductions corresponding to the ridges containing $\sigma$.
 This shows the statement about closure and completes the proof of the theorem.
\end{proof}

\subsection{Coamoeba of simple hypersurfaces}\label{S:hypersurface}

A hypersurface $H\subset\T_N(\C)$ is \demph{simple} if the support of its 
defining polynomial $f$ is an affinely
independent subset of $M$. 
Multiplying $f$ by a scalar and a monomial we may assume that it is \demph{monic} in that it 
contains the term 1, and therefore has the form
 \begin{equation}\label{Eq:simpleHypersurface}
  f\ =\ 1\ +\ a_1\xi^{\bm_1}\ +\ 
         a_2\xi^{\bm_2}\ +\ \dotsb\ +\ a_n\xi^{\bm_n}\,,
   \qquad n\leq \rank(N)\quad a_i\neq 0\,,
 \end{equation}
 where $\calA:=\{\bm_1,\dotsc,\bm_n\}$ is a  linearly independent set.
 Set $a:=(a_1,\dotsc,a_n)\in(\C^*)^n$, write $T_a$ for the translation in $(\C^*)^n$ by
 $a$, and let $\varphi_\calA\colon\T_N\to\T^n$ be the homomorphism
 \begin{equation}\label{Eq:varphiA}
   \varphi_\calA\ \colon\ x\ \longmapsto\ 
     (x(\bm_1),\,x(\bm_2),\,\dotsc,\,x(\bm_n))\,.
 \end{equation}
Then $H$ is the inverse image of the standard hyperplane~\eqref{Eq:hyperplane_simple} along the
map $T_a\circ\varphi_\calA$.

Let $\calA^\perp\subset N$ be 
$\{n\in N\mid \langle \bm,\bn\rangle=0\ \mbox{for all}\ \bm\in\calA\}$.
Then $\T_{\calA^\perp}$ is the connected component of the identity of the kernel of
$\varphi_\calA$, and so $\T_{\calA^\perp}$ acts freely on $H$.
The quotient is a simple hypersurface in $\T_N/\T_{\calA^\perp}$, which 
has the same defining polynomial~\eqref{Eq:simpleHypersurface} as $H$, but in the coordinate
ring of  $\T_N/\T_{\calA^\perp}$, which is $\C[\sat(\calA)]$.
Here, $\sat(\calA)$ is the saturation of the $\Z$-span $\Z\calA$ of $\calA$, those points
of $M$ which lie in the $\Q$-linear span of $\calA$.
The map $\varphi_\calA\colon \T_N/\T_{\calA^\perp}\to\T^n$ is surjective of degree the
cardinality of its kernel, $\Hom(\sat(\calA)/\Z\calA,\C^*)$, which we will call 
\defcolor{$\nvol(\calA)$}.
When $\sat(\calA)=\Z^n$, $\nvol(\calA)$ is the Euclidean volume of the convex hull of
$\calA\cup\{0\}$, multiplied by $n!$.

These maps $T_a$ and $\varphi_\calA$ induce linear maps on $\R_N,\R^n$ (via the
$\Log$ map defining amoebae), and $\U_N, \U^n$ (via the argument map defining coamoebae).
A consequence of this discussion is the following description of the amoeba and coamoeba
of a simple hypersurface, showing that its coamoeba may be understood largely in terms of
polyhedral combinatorics. 

\begin{lemma}\label{L:simpleH}
 Let $H$ be a simple hypersurface in $\T_N(\C)$.
\begin{enumerate}
 \item The amoeba of $H$ is the pullback of the amoeba of the standard
       hyperplane~$\eqref{Eq:hyperplane_simple}$ along the map $T_a\circ\varphi_\calA$.

 \item The coamoeba of $H$ is the pullback of the coamoeba of the standard
       hyperplane~$\eqref{Eq:hyperplane_simple}$ along the map $T_a\circ\varphi_\calA$.
       Its closure is equal to $\U_N\setminus\varphi_\calA^{-1}(T_{-a}Z)$, where 
       $Z$ is the zonotope~$\eqref{Eq:Zonotope}$.
       In particular, the number of components in the complement of $\coscrA(H)$ in $\U_N$
       is equal to $\nvol(\calA)$.
 \end{enumerate}
\end{lemma}

\begin{remark}
 As a simple hypersurface $H$ and its coamoeba are pullbacks along $T_a\circ\varphi_\calA$ of
 the standard hyperplane and its coamoeba, the statements of Theorem~\ref{T:hyperplane} hold
 for $H$, with the obvious modifications as expressed in Lemma~\ref{L:simpleH}.
 We leave the precise formulation to the reader, but note that the non-archimedean coamoeba and
 phase tropical variety of $H$ have a description that is completely combinatorial, up to
 phase shifts coming from $T_a$.
 In particular, $\Ptrop(H)$ is a polyhedral complex of pure dimension $2\dim(H)$.
\end{remark}

\begin{example}\label{Ex:simple_curve}
 Consider the simple curve $C$ in $(\C^*)^2$ defined by
\[
   f\ :=\ 1\ +\ ax^2y\ +\ bxy^2
   \qquad a,b\in\C^*\,.
\]
 Then $\calA=\{(2,1),(1,2)\}$ and $C$ is the pullback of the standard line of
 Example~\ref{Ex:AlineP2} along the composition $T_{(a,b)}\circ\varphi_{\calA}$.

 We ignore the rotation and consider only the pullback along $\varphi_\calA$.
 On $\U^2$, $\varphi_\calA$ is the map
\[
   (\alpha,\beta)\ \longmapsto\ (2\alpha+\beta,\, \alpha+2\beta)\,.
\]
 We lift to the universal cover $\R^2$ of $\U^2$ to determine the pullback.
 Figure~\ref{F:pullback} shows the image of the fundamental domain of $\U^2$
 is shaded and superimposed on nine copies of the coamoeba~\eqref{Eq:two_triangles} of the
 standard line, in the universal cover $\R^2$.
 \begin{figure}[htb]
 \[
   \raisebox{-17pt}{\begin{picture}(40,40)
     \put(0,0){\includegraphics{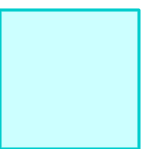}}
     \put(-20,-2){$-\pi$}  \put(-11,37){$\pi$}
     \put(-12,-11){$-\pi$}  \put(36, -11){$\pi$}
    \end{picture}}
    \quad
     \xrightarrow{\quad \varphi_\calA\quad }
    \quad
   \raisebox{-60.pt}{\includegraphics{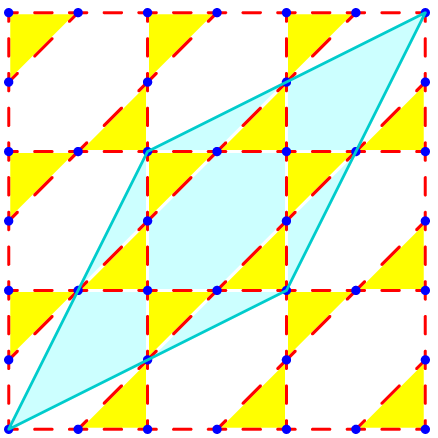}}
   \qquad\qquad\quad
   \raisebox{-45pt}{\includegraphics{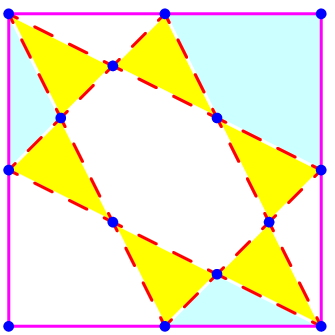}}
 \]
 \caption{The pullback and the coamoeba}
 \label{F:pullback}
\end{figure}
 Thus the pullback is the coamoeba shown in a single fundamental domain
 on the right, where we have shaded one of the three components of its complement.
 \qed
\end{example}

\subsection{Simple varieties}\label{SS:simple}

A subvariety $X\subset\T_N(\C)$ is \demph{simple} if it is the pullback along a surjective
homomorphism of a product of hyperplanes.
That is, there are hyperplanes  $H_i\subset(\C^*)^{n_i}$ for $1\leq i\leq m$ and
a surjective homomorphism 
$\varphi\colon\T_N(\C)\twoheadrightarrow(\C^*)^{n_1}\times\dotsb\times(\C^*)^{n_m}$
such that 
 \begin{equation}\label{Eq:XIsPullBack}
   X\ =\ \varphi^{-1}\bigl( H_1\times H_2\times\dotsb\times H_m\bigr)\,
 \end{equation}
Thus, up to a finite cover as in Example~\ref{Ex:simple_curve}, its coamoeba is the product of
coamoebae of hyperplanes, and therefore is an object from polyhedral combinatorics.

Simple subvarieties have a combinatorial characterization.
A polynomial $f\in\K[M]$ is \demph{monic} if its support includes $0\in M$.
The \demph{reduced support} of a monic polynomial consists of the non-zero elements of its
support. 
We note that any subscheme $X\subset\T_N(\C)$ may be defined by monic polynomials---simply
divide each polynomial defining $X$ by one of its terms, which is invertible in $\K[M]$.

\begin{proposition}\label{Prop:simple}
 A subscheme $X\subset\T_N(\C)$ is a simple subvariety if and only if it is defined by 
 monic polynomials whose reduced supports are linearly independent.
\end{proposition}

\begin{proof}
 Given a simple subvariety $X\subset\T_N(\C)$, let 
 $H_i\subset(\C^*)^{n_i}$ for $i=1,\dotsc,m$ be hyperplanes and $\varphi$ the surjection 
 $\T_N(\C)\twoheadrightarrow(\C^*)^{n_1}\times\dotsb\times(\C^*)^{n_m}$ so that 
 $X$ is the pullback of the product of the $H_i$.
 For each $i=1,\dotsc,m$, let $g_i$ be a monic polynomial defining the hyperplane
 $H_i\subset(\C^*)^{n_i}$ and set $f_i:=\varphi^*(g_i)$.
 Then $f_1,\dotsc,f_m$ are monic polynomials defining $X$.

 For each $i=1,\dotsc,n$, the reduced support of $g_i$ is a linearly independent
 subset of $\Z^{n_i}$.
 The reduced support $\calA_i$ of $f_i$ is the image of $\calA'_i\subset\Z^{n_i}$ under 
 the map
\[
   \Phi\ \colon\ \Z^{n_1}\oplus \Z^{n_2}\oplus\dotsb\oplus \Z^{n_m}
    \ \lhook\joinrel\relbar\joinrel\rightarrow\  M
\]
 which corresponds to the surjection $\varphi$.
 It follows that the reduced supports $\calA_1,\dotsc,\calA_m$ of the 
 monic polynomials $f_1,\dotsc,f_m$ defining $X$ are linearly independent.

 For the other direction, suppose that we have  monic polynomials  $f_1,\dotsc,f_m$ defining
 $X$ whose reduced supports $\calA_1,\dotsc,\calA_m$ are linearly independent.
 If  $\calA$ is the union of these reduced supports, then we obtain a surjective homomorphism
 $\varphi_\calA$~\eqref{Eq:varphiA} from $\T_N(\C)$ onto the product of 
 tori $(\C^*)^{n_0}\times(\C^*)^{n_1}\times\dotsb\times(\C^*)^{n_m}$ where 
 $n_i$ is the cardinality of $\calA_i$.
 Furthermore, each polynomial $f_i$ is the pullback of a polynomial on 
 $(\C^*)^{n_i}$ defining a hyperplane $H_i$, and so we obtain the
 description~\eqref{Eq:XIsPullBack}, which completes the proof.
\end{proof}

\begin{remark}
 Up to a finite cover (and/or a product with a torus) the coamoeba of a simple variety is the
 product of coamoebae of hyperplanes.
 By Proposition~\ref{Prop:simple}, nonempty initial schemes of simple subvarieties are simple
 subvarieties.
 Thus the structures of coamoebae, non-archimedean coamoebae, and phase tropical varieties of
 simple varieties may be understood in completely combinatorial terms by adapting the
 descriptions of Theorem~\ref{T:hyperplane} (to products of hyperplanes) and the effects of
 pullbacks along surjective homomorphisms as discussed in Subsection~\ref{S:hypersurface}.
 In particular, the phase tropical variety of a simple variety $X$ is a polyhedral complex of
 pure dimension $2\dim(X)$. 
\end{remark}

\subsection{Tropically simple varieties}\label{S:TSV}

 A $\K$-subscheme $X$ of $\T_N$ is \demph{tropically simple} if each of its tropical reductions
 $X_w$ is a simple subvariety of $\T_N(\C)$.

\begin{theorem}
 Let $X\subset\T_N$ be a tropically simple variety.
 Its non-archimedean coamoeba $\nca(X)$ is a union of polyhedra of dimension 
 $\dim(X) +1$.
 Its phase tropical variety $\Ptrop(X)$ is a polyhedral complex of pure dimension
 $2\dim(X)$.
\end{theorem}

\bibliographystyle{amsplain}
\bibliography{bibl}
\label{sec:biblio}

\end{document}